\documentclass{amsart}
\usepackage{amsfonts}
\usepackage{graphicx}
\usepackage{amscd}
\usepackage{amsmath}
\usepackage{amssymb}
\usepackage{xcolor}

\usepackage{float}
\usepackage{sidecap}

\makeatletter
\@namedef{subjclassname@2010}{%
  \textup{2010} Mathematics Subject Classification}
\makeatother

\usepackage{mathrsfs}
\usepackage{amsbsy}

\newtheorem*{acknowledgement}{Acknowledgement}

\newtheorem{corollary}{\bf Corollary}

\newtheorem{proposition}{\bf Proposition}
\newtheorem{remark}{Remark}

\newtheorem{theorem}{\bf Theorem}

\theoremstyle{definition}

\numberwithin{equation}{section}

\makeatletter
\@namedef{subjclassname@2020}{%
  \textup{2020} Mathematics Subject Classification}
\makeatother
\title{ Rotational Solitons for the Curve Shortening Flow on Revolution Surfaces}
\author{H. dos Reis$^1$}
\author{B. Leandro$^2$}
\author{R. Novais$^3$}

\address{$^{1, 2, 3}$Universidade Federal de Goi\'as, IME, 131, CEP 74690-900, Goi\^ania, GO, Brazil.}

	\email{$^1$hiuri$\_$reis@ufg.br}
	
	\email{$^2$bleandroneto@ufg.br}
	
	\email{$^3$rnovais87@discente.ufg.br}

\date{\today}
\begin{document}
    \begin{abstract}
        We present a characterization for the rotational soliton for the curve shortening flow (CSF) on the revolution surfaces of $\mathbb{R}^3$. Furthermore, we describe the behavior of such curves by showing that the two ends of each open curve are asymptotic to a parallel geodesic. 
    \end{abstract}
    
    \maketitle
    \section{Introduction}\label{intro}
        The curve shortening flow (CSF) is an evolution equation involving the geodesic curvature of a given curve in some ambient space. This flow and its generalization to higher dimensions (mean curvature flow) have been studied in material science for almost a century to model things such as cell, grain, and bubble growth. Related flows have also been used to model other physical phenomena and in image processing (cf. \cite{cao2003,Colding} and the references therein).
        
        Imagine a curve evolving by the CSF as a curve sliding across a surface like an elastic band so that the instantaneous velocity at each point is proportional to the geodesic curvature of the curve at that point. Geodesics have zero geodesic curvature, so they do not move. Naturally, it is to be imagined that some closed curves will evolve towards geodesics and others will evolve towards collapsing at a point \cite{Gage}. On the other hand, open curves could expand \cite{halldorsson2012}.    
        
        Formally, a map $\widehat{\Phi} : U \times I\rightarrow \mathbb{M}^{2}\subseteq\mathbb{R}^{3}$,  $I=[0,\,\mathbf{T})$ and $U\subseteq\mathbb{R}$,  which is a family of curves, is a solution of curve shortening flow if     
        \begin{equation}\label{csf}
            \langle\partial_{t}\widehat{\Phi}^{t}(s),\eta^{t}(s)\rangle={\kappa}^{t}(s),
        \end{equation}
        where $\widehat{\Phi}^{t}(s)=\widehat{\Phi}(s,\,t)$, $s\in U$ and $t\in I$, with the initial curve $\widehat{\Phi}^{0}:=\Phi$. Moreover, $\widehat{\eta}^{t}$ and $\widehat{\kappa}^{t}$ stand for the unit normal vector and the geodesic curvature of $\widehat{\Phi}^t(s)$, respectively. When $\Phi$ is a geodesic, the family $\widehat{\Phi}^t$ gives a trivial solution to the curve shortening. This paper considers a special type of CSF by imposing that the $\widehat{\Phi}^t$ curves evolve by isometries. Equivalently, we say that $\widehat{\Phi}^{t}$ is a soliton solution for the CSF (cf. \cite{Hungerbuhler}) if there exists a Killing vector field with associated flow $\widehat{\Phi}^{t}$ such that
        \begin{eqnarray}\label{5}
            \widehat{\Phi}^{t}(s)=\Psi^{t}\left(\Phi(s)\right).
        \end{eqnarray}
        Solitons are relevant because they represent a class of solutions with very special properties, e.g., they appear as blow-ups of singularities of the mean curvature flow.
        
        In $\mathbb{R}^2$, self-similar solutions for the CSF were already classified by Halldorsson in \cite{halldorsson2012}. However, little is known about such curves when they evolve on surfaces other than the plane. In \cite{halldorsson2015}, the author classified the self-similar solutions to the mean curvature flow in $\mathbb{R}^{1,1}$. Recently, the soliton solutions for the CSF were also classified in the sphere by Tenenblat and dos Reis \cite{Hiuri1}. Although these works classified CFS solutions, they did not provide explicit solutions. A classification of such solutions in $2$-dimensional hyperbolic space was provided in \cite{Silva2}. Explicit solutions for the CSF are rare. In \cite{Silva2021}, the authors explicitly provided soliton solutions for the curve shortening flow on the light cone. Nonetheless, a well-known explicit soliton example is the {\it Grim Reaper} solution, the graph of the function $f(x)=-\ln\cos(x)$ in the $xy$-plane, where $x\in(-\pi/2,\,\pi/2)$.     
        
        {\it Does every family of curves satisfying equation \eqref{csf} which has a subsequence converging to a geodesic converge uniquely to this geodesic?} The answer to this question is affirmative for the sphere (cf. \cite{Gage,Hiuri1}). Moreover, it is well-known that uniqueness is guaranteed if the surface in the neighborhood of the geodesic has negative curvature (cf. \cite{Silva2,woolgar}). However, as far as we know, the question of geodesics on surfaces with mixed Gauss curvature remains open.
        
        Here, we prove that the two ends of a rotational soliton solution for the curve shortening flow on a revolution surface are asymptotic to the parallel geodesics. They are asymptotic to the geodesics orthogonal to the axis of rotation, i.e., the $z$-axis. We are considering open and closed solitons, respectively, Theorem \ref{theoconverge} and Corollary \ref{teoexistence}.
        
        Among several interesting results proved by Garcke and N\"urnberg, in \cite{Garcke} they presented variational approximations of boundary value problems for curve shortening flow and curve straightening flow in two-dimensional Riemannian manifolds that are conformally flat.
        
        Up to isometries of $\mathbb{R}^3$, we can consider the revolutions surface $\mathbb{M}^{2}$ in $\mathbb{R}^{3}$ rotating along the $z-$axis. Consider a generating curve $\alpha(v)=(\phi(v),0,\,\psi(v))$ such that $\phi(v)>0$ (cf. \cite{CARMO GD}) given by 
       \begin{equation}\label{ROT}
            X(u,\,v)=(\phi(u)\cos(v),\,\phi(u)\sin(v),\,\psi(u)),
        \end{equation}
        where $\phi$ and $\psi$ are smooth functions, in which $u\in I\subseteq \mathbb{R}$ and $0<v<2\pi$. The metric tensor is given by
        
        \begin{eqnarray}\label{metrictensor}
            g=du^{2}+\phi^{2}(u)dv^{2}.
        \end{eqnarray}  
        
        \begin{remark}
       If $I$ is bounded at any of its extremes, we ask that $\phi$ and $\psi$ must be identically zero at this point. Thus, the tangent plane will be well-defined at these points.
        \end{remark}

        Rotations around the $z$-axis are isometries in $\mathbb{M}^2$. For this reason, we will consider the rotational solitons, i.e., solutions to CSF that rotate around the axis of the surface of revolution. Moreover, it is worth pointing out that a soliton $\Phi(s)=X(u(s),\,v(s))$ of the CSF in ambient space $\mathbb{M}^2$ is defined for every real parameter since $\mathbb{R}^2$ is the covering space for the revolution surface $\mathbb{M}^{2}$. 
        
        We now state our main results. The following result characterizes the solitons on the revolution surface. Similar results for CSF solitons in space forms, Lorenzian two-space, and light cone can be found in \cite{halldorsson2012,halldorsson2015,Hiuri1,Silva2021} and \cite{Silva2}. It is interesting to point out that we are considering a particular case of \cite[Definition 3.1]{alias}.

        \begin{theorem}\label{caracROT}
        Let $\Phi:I\subseteq \mathbb{R}\rightarrow \mathbb{M}^{2}$ be a regular curve  parametrized by arc length. Then $\Phi(s)=X(u(s),\,v(s))$ is a revolution soliton to the curve shortening flow if and only   
        \begin{eqnarray}\label{kgeoROT}
        \kappa(s)=a\phi(s) u'(s),
        \end{eqnarray}
        where $\kappa$ is the geodesic curvature of $\Phi$ and $a$ is a constant. If $a=0$ we have a geodesic. 
        \end{theorem}

        This result is similar to the characterization provided in \cite{Hiuri1,Silva2021} and \cite{Silva2}. The following theorem describes the behavior at infinity of the ends of the rotational solitons in a surface of revolution. 
        
        \begin{theorem}\label{theoconverge}
        Let $\Phi:I \subseteq \mathbb{R} \rightarrow \mathbb{M}^{2} \subset \mathbb{R}^{3}$ be a curve parametrized by the arc length. If $\Phi(s)=X(u(s),\,v(s))$ is a revolution soliton to the CSF with bounded total geodesic curvature, $|\phi|<\infty$ and $\ |\dot{\phi}| <\infty$, then each end of the curve are asymptotic to the parallel geodesic.
        \end{theorem}

        \begin{remark}
            By bounded total geodesic curvature we are referring to $$\|\kappa(s)\|^2=\int_I\kappa(s)^2<\infty,$$
            where $\kappa(s)$ is the geodesic curvature of $\Phi(s).$
        \end{remark}   
        
        The next result is a consequence of Theorem \ref{theoconverge}. 
        
        \begin{corollary}\label{teoexistence}
           Let $\Phi$ be the revolution soliton of the CSF on $\mathbb{M}^2$, with $\Phi$ integrable. If $\Phi$ is a simple closed curve, then  $\Phi$ is a parallel geodesic.
        \end{corollary}  

        In Section 4, we will apply our results in the torus $\mathbb{T}^2$. 
        We emphasize that this case is interesting because the torus does not have a defined curvature sign and is a surface of revolution given by $X(u,\,v)=(\phi(u)\cos(v),\,\phi(u)\sin(v),\,\psi(u))$, where the generating curve $\alpha(v)=(\phi(u),\,0,\,\psi(u))$ satisfies the condition $|\phi|, |\dot{\phi}| < \infty$. Furthermore, the only isometry is the rotations around the $z-$axis.  

        In Section 5, we will show that catenoids do not satisfy the condition $|\phi|, |\dot{\phi}| < \infty$. Therefore, Theorem \ref{theoconverge} does not apply to the catenoid. This means that we do not necessarily have solutions asymptotic to the central equator of the catenoid. This fact highlights the importance of the hypothesis over the generating curve for the revolution surface.

    \section{Background}\label{back}

         Throughout this work, we will always consider a revolution surface $\mathbb{M}^2$ within (\ref{ROT}). 
        In $\mathbb{M}^{2}$, let's consider the following orthonormal frame:    

         \begin{eqnarray}\label{refROT}
            e_1 &=&(\dot{\phi}\cos(v),\,\dot{\phi}\sin(v),\,\dot{\psi});\nonumber\\
            e_2 &=& (-\sin(v),\ \cos(v), \ 0);\\
            N(u,v) &=& (-\dot{\psi}\cos v,\,-\dot{\phi}\sin v,\, \dot{\phi}).\nonumber
        \end{eqnarray}

        To simplify the notation, we will use

        \begin{align*}
        \dot{\phi}=\dfrac{d \phi}{du}, && \dot{\psi}=\dfrac{d \psi}{du}, && \ddot{\phi}=\dfrac{d^2 \phi}{du^2}, && \ddot{\psi}=\dfrac{d^2 \psi}{du^2}.
        \end{align*}
        Moreover, the derivative with respect to the parameter $s$ or $t$ will be denoted by $'$. For instance,
        \begin{align*}
        u'(s) = \dfrac{d}{ds}u(s)\quad\mbox{and}\quad \xi'(t)=\dfrac{d}{dt}\xi(t).
        \end{align*}

        In terms of the frame \ref{refROT}, the coefficients of the first and second fundamental forms are given by 

        \begin{align}\label{dX}
        X_{u}= e_1, && X_{v}= \phi e_2, \nonumber\\
        \\ 
        X_{uu}=-(\ddot{\phi}\dot{\psi}-\ddot{\psi}\dot{\phi})N,&& X_{uv}=\dot{\phi} e_{2} && X_{vv}=-\phi\dot{\phi}e_1+\phi\dot{\psi}N,\nonumber.
        \end{align}
        
        Considering $\mathbb{M}^2$ the ambient space, let $\Phi(s)=X(u(s),\,v(s))$ a curve parameterized by length of the arc, then

        \begin{eqnarray}\label{TROT}
         T=u' X_u +v' X_v=u'e_{1}+\phi v'e_{2}
        \end{eqnarray}
        
        and
        \begin{eqnarray}\label{etaROT}
         \eta=N \wedge T= -v' \phi e_{1}+ u' e_{2},
        \end{eqnarray}
        where $T=\Phi'$ is the unit tangent vector field  and $\eta$ is the normal tangent vector field $\Phi$ with respect to $\mathbb{M}^2$.  
        
        The vector fields $\{ T, \ \eta, \ N \}$ form the Darboux frame such that 
        \begin{eqnarray}\label{darbouxROT}
        T'=\kappa\eta+\kappa_{n}N,\quad\eta'=-\kappa T+\tau N\quad\mbox{e}\quad N'=-\kappa_{n}T-\tau\eta,
        \end{eqnarray}
        were $\kappa$, $\kappa_{n}$, and $\tau$ are the geodesic curvature, the normal curvature, and the geodesic torsion of the curve $\Phi$, respectively.

        Taking the derivative and using (\ref{refROT}), we have
        \begin{eqnarray*}
        T'&=& u''X_{u}+ v''X_{v}+2u'v'X_{uv}+(u')^{2}X_{uu}+(v')^{2}X_{vv}\nonumber\\
        &=&[ u''-(v')^2\phi\dot{\phi}]e_{1} + [v''\phi+2u'v'\dot{\phi}]e_{2}+[-(u')^2(\ddot{\phi}\dot{\psi}-\dot{\phi}\ddot{\psi})+(v')^{2}\phi\dot{\psi}]N.\nonumber
        \end{eqnarray*}
        From the equation \eqref{darbouxROT}
    
        \begin{eqnarray}\label{curvdarbousROT}
        \kappa = \langle T',\eta\rangle=(- u'' v'+v''u')\phi+v'(1+(u')^{2})\dot{\phi},
        \end{eqnarray}
        where we used that the curve is parameterized by length arc, i.e.,
        \begin{eqnarray}\label{normTROT}
         (u')^{2}+(v')^{2}\phi^2=1.
        \end{eqnarray}
        Taking the derivative of \eqref{normTROT}, we have

        \begin{eqnarray}\label{DnormTROT}
         u'u''+v'v''\phi^2+u'(v')^2\phi\dot{\phi}=0.
        \end{eqnarray}

        A regular curve parametrized by arc length $\Phi(s)=X(u(s),v(s))$ on $\mathbb{M}^2$ is a rotation soliton to curve shortening flow if there is a one-parameter family of rotations around (isometries) the axis of revolution ($z$-axis), i.e.,
        \begin{equation}\label{Rot}
        R(t)=\left ( \begin{array}{ccc}
         \cos\xi(t) & -\sin\xi(t) & 0\\
        \sin\xi(t) & \cos\xi(t) & 0\\
         0 & 0 & 1\\
         \end{array} \right ),
        \end{equation}
        where $\xi$ is a smooth function such that $\xi(0) = 0$. Then,
         \begin{eqnarray}\label{rotsol}
         \widehat{\Phi}^{t}(s)=R(t)\Phi(s)
        \end{eqnarray}
        is a solution to curve shortening flow, i.e., satisfies Equation \eqref{csf}. The condition $\xi(0)=0$ ensures that $\phi(s)$ is the initial condition of the CSF.
    \section{Rotational Solitons for the CSF on Revolution Surfaces \\ Main Result}\label{main}
    This section is reserved for the demonstration of the main results of this manuscript.
     \begin{proof}[{\bf Proof of Theorem \ref{caracROT}}]
    Suppose that $\Phi(s)$ is parametrized by arc length and a rotational soliton solution to the CSF in $\mathbb{M}^{2}$. Hence, $$\widehat{\Phi}^t(s)= R(t)\Phi(s)$$ is a solution to the CSF, where $R(t)$ is given by \eqref{Rot}.
    
    Taking the derivative in the above equation at $t=0$, we have 
        	\begin{eqnarray*}
        	  {\displaystyle \partial_t \widehat{\Phi}^{0}(s)=R'(0)\Phi(s)=\xi'(0) \phi e_{2}.}
        	\end{eqnarray*}
            Making $a=\xi'(0)$, by the above identity and \eqref{etaROT}, from \eqref{csf} we get \eqref{kgeoROT}.
        	
        	Conversely, let $\Phi(s)=X(u(s),v(s))$ be a curve in $\mathbb{M}^{2}$ parametrized by arc length $s$ such that 
            $$\kappa(s)=a\phi(s) u'(s).$$

            Define $\widehat{\Phi}^t(s)=R(t)\Phi(s)$, where
            $$R(t)=\left ( \begin{array}{ccc}
            \cos at & -\sin at & 0\\
            \sin at & \cos at & 0\\
            0 & 0 & 1\\
            \end{array} \right ).$$

            We will show that $\widehat{\Phi}^t(s)$ is a solution to the CSF on $\mathbb{M}^2$. To that end, since $R(t)$ is an isometry on $\mathbb{M}^{2}$, we have that $\widehat{\kappa}^{t}(s)=\kappa(s)$, $\widehat{\eta}^{t}(s) = R(t) \eta(s)$  for all $t\in I$. Thus,
            $$\widehat{\eta}^{t}(s)= \left( -v'\phi\dot{\phi}\cos(at+v) -u'\sin(at+v), -v'\phi\dot{\phi}\sin(at+v)+u'\cos(at+v), -v'\phi\dot{\psi} \right).$$
            Taking the derivative of  $\widehat{\Phi}^t(s) = R(t)\Phi(s)$ with respect to $t$, we have
              
              $$\displaystyle \partial_t \widehat{\Phi}^{t}(s)=\partial_t R(t)\Phi(s)=a \phi\left( -\sin(at+v), \cos (at+v), 0 \right).$$
            Therefore,
            \begin{eqnarray*}
                \left\langle\displaystyle \partial_t \widehat{\Phi}^{t},\widehat{\eta}^{t}\right\rangle 
                                = a\phi u' = \widehat{\kappa}^{t}.
            \end{eqnarray*}     
      \end{proof}   

We will derive an ODE system for the rotational soliton of the CSF on a revolution surface. Using this system, we prove our main result.

      \begin{proposition}\label{teoEDO}
        	Let $\Phi:I \subseteq \mathbb{R}\rightarrow \mathbb{M}^{2}\subset\mathbb{R}^{3}$ be a curve parametrized by the arc length on the revolution surface. Then, $\Phi(s)=X(u(s),\,v(s))$ is a soliton for the CSF on the revolution surface if and only if
             \begin{eqnarray}
                \begin{cases}\label{sistemaROT}
                    u''=-au'v'\phi^2+ (v')^2\phi\dot{\phi};\\ \\
                    v''=a(u')^2-\dfrac{2u' v'\dot{\phi}}{\phi}.\\
                \end{cases}
            \end{eqnarray}   
            Moreover, if $a=0$ we obtain the geodesic ODE system for $\mathbb{M}^2$.
        \end{proposition}   

        \begin{proof}[Proof of Proposition \ref{teoEDO}]
        Combining Theorem \ref{caracROT} with \eqref{curvdarbousROT} we obtain
        $$( - u'' v'+u'v'')\phi+ v'(1+(u')^{2})\dot{\phi}=au'\phi.$$
        So,
        \begin{eqnarray}\label{curvaturaAUX}
        - u'' v'+u'v'' =au' - v'(1+(u')^{2})\frac{\dot{\phi}}{\phi}.
        \end{eqnarray}
        Then, multiplying \eqref{DnormTROT} by $u'$ and \eqref{curvaturaAUX} by $ v'\phi^2$ yields to
        $$(u')^2 u''=-(u')^2(v')^2\phi\dot{\phi}-u'v'v''\phi^2$$     
        and    
        $$(v')^2\phi^2 u''=-au'v'\phi^2 + (v')^2(1+(u')^{2})\phi\dot{\phi}+ u'v'v''\phi^2 .$$
        Combining the last two equations and using equation \eqref{normTROT}, we obtain
        $$u''=-au'v'\phi^2+ (v')^2\phi\dot{\phi}.$$

        Now, we multiply \eqref{DnormTROT} by $v'$ and the equation \eqref{curvaturaAUX} by $u'$ 
        to obtain, respectively,
        $$u'v'u''+(v')^2v''\phi^2=-u'(v')^3\phi\dot{\phi}$$
        and
        $$- u'' v'u'+(u')^2v'' = a(u')^2 - u'v'(1+(u')^{2})\frac{\dot{\phi}}{\phi}.$$
        From the last two equations, and again using \eqref{normTROT} we have
        $$v''=a(u')^2-\dfrac{2u' v'\dot{\phi}}{\phi}.$$
        \end{proof}

        To prove our next result, let us define the following function

        \begin{eqnarray}\label{definicaoLROT}
            \Lambda(s)= \dfrac{u'(s)}{\sqrt{1-(u'(s))^2}}.\nonumber
        \end{eqnarray}
        
         Since $\Phi(s)$ is a curve parametrized by the arc length, $\Lambda$ is well-defined since $1-(u')^{2}=(v')^{2}\phi^{2}\geq 0$; see \eqref{normTROT}. Moreover, $1-(u')^{2}=0$ if, and only if $v'=0$ or is an extreme point of the generating curve. In the following, we will assume that the set of parameters for which $1-(u')^2 = 0$ is not dense. In an open set where $1-(u')^2 = 0$, the curve $\phi$ is a geodesic. However, geodesics are trivial solutions to the CSF, and we are not interested in this case. Under this hypothesis, we can consider that $1-(u')^{2}>0$. In this case, $u'$ is a limited function with $|u'|<1$.

         \begin{theorem}\label{maintheorem1}
	        Let $\Phi:I\subseteq\mathbb{R}\rightarrow\mathbb{M}^{2}\subset\mathbb{R}^{3}$ be a curve parametrized by the arc length on the revolution surface. If $\Phi(s)=X(u(s),\,v(s))$ is an rotational soliton for the CSF in $\mathbb{M}^2$, then
            \begin{eqnarray}\label{normaL2}
                \kappa(s) =\pm \Lambda(s)\left(c + \int_{I}\kappa(s)^2\right);\quad c\in\mathbb{R},
            \end{eqnarray}
            where $\kappa$ stands for the geodesic curvature of $\Phi$.
        \end{theorem}      

        \begin{proof}[Proof of Theorem \ref{maintheorem1}]
        Combining equations  \eqref{normTROT}  and \eqref{DnormTROT} we get
        \begin{equation*}
        v''=-\dfrac{\dot{\phi}}{\phi}u' v'- \dfrac{u'v'u''}{1-(u')^2}.
        \end{equation*}
        Then, from Proposition \ref{teoEDO} and the above equation, we obtain 
        $$v''=-a(u')^2- \dfrac{2u'v'u''}{1-(u')^2}.$$
        We can rewrite the above ODE as follows
         $$\dfrac{v''[1-(u')^2]-v'[1-(u')^2]'}{[1-(v')^2]^2}=\dfrac{-a(u')^2}{1-(u')^2},$$
        which is equivalent to
        \begin{equation}\label{DauxROT}
                \dfrac{d}{ds}\left(\dfrac{v'}{1-(u')^2}\right)=\dfrac{-a(u')^2}{1-(u')^2}.
        \end{equation}
        From \eqref{normTROT} we have
         $$v'=\dfrac{\pm\sqrt{1-(u')^{2}}}{\phi}.$$
         Now, Theorem \ref{caracROT} gives us
         $$\kappa v'=\mp a u' \sqrt{1-(u')^{2}},$$
        Thus,
        $$\dfrac{v'}{1-(u')^{2}}=\mp\dfrac{ a u'}{\kappa\sqrt{1-(u')^{2}}}.$$
        Furthermore, from \eqref{DauxROT} we get
        \begin{eqnarray*}
                \dfrac{d}{ds}\left(\mp\dfrac{ a u'}{\kappa\sqrt{1-(u')^{2}}}\right) &=& \dfrac{d}{ds}\left(\dfrac{v'}{1-(u')^{2}}\right)\\
                &=& \dfrac{-a(u')^2}{1-(u')^2}\\
                &=& -a\left(\dfrac{u'}{\sqrt{1-(u')^2}}\right)^2.
            \end{eqnarray*}
         Considering $\Lambda(s)=\dfrac{u'}{\sqrt{1-(u')^2}}$ the above identity becomes
        \begin{equation}\label{lambdaAUX}
                \pm \dfrac{d}{ds}\left(\dfrac{\Lambda}{\kappa}\right)=\Lambda^2.
        \end{equation}
        Therefore, 
        \begin{eqnarray*}
                 \Lambda^2 &=& \pm \dfrac{d}{ds}\left(\dfrac{\Lambda}{\kappa}\right);\\
                 \kappa^2 &=& \pm\dfrac{\dfrac{d}{ds}\left(\dfrac{\Lambda}{\kappa}\right)}{\left(\dfrac{\Lambda}{\kappa}\right)^2};\\
                 \kappa^2 &=& \mp\dfrac{d}{ds}\left[\left(\dfrac{\Lambda}{\kappa}\right)^{-1}\right].\\
            \end{eqnarray*}
         By integration we get
           \begin{eqnarray*}
                \kappa_{g} = \pm \Lambda\left(c + \int_{I}\kappa_{g}^2\right),
            \end{eqnarray*}
        where $c\in\mathbb{R}$.
       \end{proof}

       \begin{proof}[{\bf Proof of Theorem \ref{theoconverge}}]
       Taking the derivative of the geodesic equation in Theorem \ref{caracROT}, we have
        \begin{equation}\label{kappa'}
           \kappa'=a\left(\dot{\phi}(u')^2+\phi u''\right).
        \end{equation}
        
         Multiplying the previous equation by $\dfrac{\Lambda}{\kappa}$ and combining it with $$\Lambda'=\dfrac{u''}{[1-(u')^2]^{3/2}}$$ gives us
         \begin{eqnarray*}
               \ \frac{\Lambda}{\kappa}\kappa' &=& \dfrac{a u'}{\kappa\sqrt{1-(u')^2}}\left[\dot{\phi}(u')^2+\phi u''\right]\\
                                                      &=& \dfrac{1}{\sqrt{1-(u')^2}}\left[\dfrac{\dot{\phi}}{\phi}(u')^2+ u''\right]\\
                                                      &=& \dfrac{\dot{\phi}(u')^2}{\phi\sqrt{1-(u')^2}} + \dfrac{u''}{\sqrt{1-(u')^2}}\\
                                                      &=& \dfrac{\dot{\phi}(u')^2}{\phi\sqrt{1-(u')^2}} + \Lambda'[1-(u')^2].
            \end{eqnarray*}    
        From the equation \eqref{lambdaAUX} we can see that $$\kappa'\frac{\Lambda}{\kappa}=\Lambda'\pm \kappa{\Lambda}^2.$$
        Combining the last two identities we have
        \begin{eqnarray*}
               \Lambda'\pm \kappa{\Lambda}^2 &=& \dfrac{\dot{\phi}(u')^2}{\phi\sqrt{1-(u')^2}} + \Lambda'[1-(u')^2]\\
                                  \Lambda'(u')^2 &=& \mp \kappa{\Lambda}^2+\dfrac{\dot{\phi}(u')^2}{\phi\sqrt{1-(u')^2}}\\
                                        \Lambda' &=& \mp \dfrac{\kappa}{(u')^2}\left({\dfrac{u'}{\sqrt{1-(u')^2}}}\right)^2+\dfrac{\dot{\phi}}{\phi\sqrt{1-(u')^2}}\\
                                        \Lambda' &=& \mp {\dfrac{\kappa}{1-(u')^2}}+\dfrac{\dot{\phi}}{\phi\sqrt{1-(u')^2}}\\
                                        \Lambda' &=& \pm {\dfrac{a \phi u'}{1-(u')^2}}+\dfrac{\dot{\phi}}{\phi\sqrt{1-(u')^2}}\\
                   \dfrac{u''}{[1-(u')^2]^{3/2}} &=& \pm {\dfrac{a \phi u'}{1-(u')^2}}+\dfrac{\dot{\phi}}{\phi\sqrt{1-(u')^2}}\\
                   \dfrac{u''}{1-(u')^2} &=& \pm {\dfrac{a \phi u'}{\sqrt{1-(u')^2}}}+\dfrac{\dot{\phi}}{\phi}.
            \end{eqnarray*}    
         Therefore, 
         \begin{equation}\label{u''}
             |u''|\leq |a|\phi+\dfrac{|\dot{\phi}|}{\phi}.
         \end{equation}
          
       Define $$f(s)= \int_{s_0}^{s}\kappa_{g}^2.$$ 
       Note that $f' = \kappa^2 \geq 0$, so $f$ is non-decreasing. Since $f$ is finite, the limit $\displaystyle\lim_{s\rightarrow\infty}f(s)$ exists.
        On the other hand, $f''=2\kappa'\kappa.$ Then, a straightforward computation from \eqref{kgeoROT}, \eqref{kappa'} and \eqref{u''} proves that $|f''|$ is bounded, and thus $f'$ is uniformly continuous.
        
        Therefore, we can apply a Lyapunov-type lemma \cite[Lemma 4.3]{slotine} to get 
        \begin{equation}\label{Eqlim}
        \displaystyle\lim_{s\rightarrow\infty}f'(s)=\displaystyle\lim_{s\rightarrow\infty}\kappa^2=0.
        \end{equation}
        To show that the ends of $\Phi(s)=X(u(s),v(s))$ are asymptotic to geodesics, we define $\theta$ the angle between the $\Phi$ curve and the meridians of $\mathbb{M}^2$ by
         \begin{eqnarray*}
                \cos\theta= \frac{1}{|X_u|}\langle X_u,\, T\rangle.
        \end{eqnarray*}
        Use \eqref{dX} and \eqref{TROT} in the above equation to obtain
        \begin{eqnarray*}
                \cos\theta= u'.
            \end{eqnarray*}
       Theorem \ref{caracROT} combined with \eqref{Eqlim} leads us to $\lim_{s\rightarrow\infty}u'=0$. Then,
    \begin{eqnarray*}
                \lim_{s\rightarrow\infty}\cos\theta= \lim_{s\rightarrow\infty}u'=0.
    \end{eqnarray*}
    This demonstrates that the ends of $\Phi$ are perpendicular to the meridians. Thus, $\Phi$ must be asymptotic to the parallels which are geodesics. 
    \end{proof}

     \begin{proof}[{\bf Proof of Corollary \ref{teoexistence}}]
         Let $h(u)$ be a function such that $h'(u)=\phi(u)$.
            Suppose that the region $R$ bounded by $\Phi$ is simple.
         From the Gauss-Bonnet Theorem, we have
         $$\int_{\partial R}\kappa ds + \int\int_{R}Kd\mathrm{a} =2\pi.$$
        Moreover, from the formulae for the first variation of area $\mathrm{A}$, we know that
         $$\int_{\partial R}\kappa ds + \dfrac{d\mathrm{A}}{dt} = 0.$$
         Since isometries preserve the area, we must have
         $\mathrm{A}(t)$ constant, then $$\int_{\partial R}\kappa ds = 0.$$
         By the characterization of the geodesic curvature given by Theorem \ref{caracROT}, we conclude that
         \begin{eqnarray*}
                0=\int_{\partial R}\kappa  ds = a\int_{\partial R} \phi u'ds = ah(u) + c,
            \end{eqnarray*}
          with $c\in\mathbb{R}$. So,
           \begin{eqnarray*}
                0=ah(u) + c
           \end{eqnarray*}
          Taking the derivative we get
           \begin{eqnarray*}
                0=a\phi u'.
           \end{eqnarray*}
           Since $\phi>0$, we have $u'=0$, i.e., $\Phi$ is a parallel geodesic of $\mathbb{M}^2$.
     \end{proof}

    We highlight here the importance of the conditions $|\phi|, |\dot{\phi}| < \infty$ in Theorema \ref{theoconverge}. We can consider the plane as a surface of revolution of $\mathbb{R}^3$ where $\phi(u)=u$ and $\psi(u)=0$. In \cite{halldorsson2012}, Halldorsson provided examples of revolution solitons in $\mathbb{R}^2$ where the ends of such solutions are unbounded. Another example that proves the necessity of this hypothesis is the catenoid with generating curves $$(\phi(u),\,0,\,\psi(u))=(\cosh u,\,0,\,u)$$  which do not satisfies the condition $|\phi|, |\dot{\phi}| < \infty$. Therefore, Theorem \ref{theoconverge} can not be applied to the catenoid. This means we may not have solutions asymptotic to the catenoid's central equator (see numerical solution in Appendix \ref{appendix}). This fact emphasizes the importance of the conditions on the generating curve of a revolution surface.

    \section{Soliton for the CSF on Torus $\mathbb{T}^2$}\label{torus}
        This section is reserved to study soliton solutions for the CSF on $\mathbb{T}^2.$ Consider the torus $\mathbb{T}^{2}$ in $\mathbb{R}^{3}$ (cf. \cite{CARMO GD}) given by revolution surface \eqref{ROT}, where
        \begin{eqnarray}
         \phi(u) & = &  R+r\cos u \nonumber \\   
        \psi(u) & = & r\sin u,\ \text{where} \  0<r<R. 
        \end{eqnarray}

         Let $\zeta$ be a tangent vector field in $\mathbb{T}^{2}$ given by
        \begin{eqnarray*}
        	\zeta=\zeta^{1}(u,\,v)X_{u}+\zeta^{2}(u,\,v)X_{v},
        \end{eqnarray*}
        where $\zeta^{i}$ are smooth functions. Remember that
        \begin{eqnarray}\label{td1}
            X_{u}(u,\,v)=\left(-r\sin u\cos v,\,-r\sin u\sin v,\,r\cos u\right)
        \end{eqnarray}
        and
        \begin{eqnarray}\label{td2}
            X_{v}(u,\,v)=\left(-(R+r\cos u)\sin v,\,(R+r\cos u)\cos v,\,0\right).
        \end{eqnarray}
    
        Now, if $\zeta$ is a Killing vector field then
        \begin{eqnarray}\label{killingeq}
            \zeta^{k}g_{ij\,;k}+\zeta^{k}_{\,;i}g_{jk}+\zeta^{k}_{\,;j}g_{ik}=0.
        \end{eqnarray}
        Therefore, from \eqref{metrictensor} we obtain (cf. (4.6.12) in \cite{fecko}) that the only solution for \eqref{killingeq} is given by the generator of rotations around the $z$-axis, i.e., $\zeta^{1}=0$ and $\zeta^{2}=A$, where $A$ is constant.  
        
        Considering the torus $\mathbb{T}^2$ by $X(u,\,v)=\left(x(u,v),\,y(u,v),\,z(u,v)\right)$ and the canonical frame of $\mathbb{R}^3$ by $\{\partial_{1},\partial_{2},\partial_{3}\}$, we have
        \begin{eqnarray*}
        	\zeta=A(-y\partial_{1}+x\partial_{2}).
        \end{eqnarray*}
        Hence, if $\Psi^{t}$ is an $1$-parameter family of isometries of $\mathbb{T}^{2}$ then 
        \begin{eqnarray*}
        	\left\{
        	\begin{array}{lll}
        		\dfrac{d}{dt}\Psi^{t}(p)=\zeta(\Psi^{t}(p));\\\\
        		
        		\Psi^{0}(p)=p;\quad p\in\mathbb{T}^{2}.
        	\end{array}
        	\right.
        \end{eqnarray*}
        A straightforward computation gives us
        \begin{eqnarray*}
        	\Psi^{t}(x,\,y,\,z)=(x\cos t - y\sin t,\,x\sin t+ y\cos t ,\, z),
        \end{eqnarray*}
        where 
        \begin{eqnarray*}\label{param}
            \left(R - \sqrt{x^2 + y^2}\right)^{2} + z^2 = r^2.
        \end{eqnarray*}
    
        From now on, we will denote by $\Phi$ any curve parameterized by the arc length on $\mathbb{T}^{2}$. Writing $\Phi(s)=(x(s),\,y(s),\,z(s))$, we have
        \begin{equation}\label{csf in T2}
            \Phi^{t}(s)=\Psi^{t}(\Phi(s))=\left ( \begin{array}{ccc}
            \cos\xi(t) & -\sin\xi(t) & 0\\
            \sin\xi(t) & \cos\xi(t) & 0\\
            0 & 0 & 1\\
            \end{array} \right )\left ( \begin{array}{ccc}
            x(s) \\
            y(s) \\
            z(s) \\
            \end{array} \right ),
        \end{equation}
        where $\xi$ are smooth functions such that $\xi(0)=0$. Therefore, the isometries of $\mathbb{T}^{2}$ generated by Killing fields are given by \eqref{Rot}. This result shows that Theorem \ref{theoconverge} describes all soliton solutions of the CSF on the torus. Then, as a consequence of the Theorem \ref{theoconverge} we have the following result.

         \begin{theorem}
	       Let $\Phi: I\subseteq\mathbb{R}\rightarrow\mathbb{T}^{2}\subset\mathbb{R}^{3}$ be an arc-length parametrized curve. If $\Phi(s) = X(u(s), v(s))$ is a soliton of the CSF on $\mathbb{T}^2$ with bounded total geodesic curvature, then $\Phi$ is asymptotic to the equators.
        \end{theorem}

         We already know that if $\kappa=0$ we have geodesics in $\mathbb{T}^2$, trivial solutions for the CSF. Inspired by this, we prove a result of the nonexistence of soliton solutions with constant curvature $\kappa$ for the CSF on $\mathbb{T}^2.$
        \begin{corollary}\label{maintheorem2}
	        There is no soliton solution for the CSF on $\mathbb{T}^2$ with non-null constant geodesic curvature.
        \end{corollary}

        Now, we will derive an ODE system \eqref{teoEDO} for the soliton solutions of the CSF on $\mathbb{T}^2$. Then, by using this system, we prove some numerical approximations for some solutions. 

         \begin{proposition}
        	Let $\Phi:I\subseteq\mathbb{R}\rightarrow\mathbb{T}^{2}\subset\mathbb{R}^{3}$ be a curve parametrized by the arc length on the Torus. Then, $\Phi(s)=X(u(s),\,v(s))$ is a soliton for the CSF on the torus if and only if
            \begin{eqnarray*}
                \begin{cases}
                    v''=ar^{2}(u')^2+\dfrac{2u' v'r\sin u}{R+r\cos u};\\ \\
                    u''=- v'(R+r\cos u)\left[\dfrac{ v'\sin u}{r}+a(R+r\cos u)u'\right].\\
                \end{cases}
            \end{eqnarray*}  
            Moreover, if $a:=\xi'(0)=0$ we obtain the geodesic ODE system for $\mathbb{T}^2$.
        \end{proposition}

        \section{Appendix: A gallery of soliton solutions for the CSF on $\mathbb{T}^2$ and revolution catenoid}\label{appendix}
        This section is reserved to prove some numerical soliton solutions for the CSF on $\mathbb{T}^2$. To that end, we use the system provided in Proposition \ref{teoEDO}. We can conclude that those solutions are asymptotic to the outer equator.

        Furthermore, we show numerical solutions to the revolution surface generated by the revolution of the curve $(\cosh u,\,0,\,u)$ around the $z$-axis. Note that the generating curve of the catenoid does not satisfy the condition $|\phi|, |\dot{\phi}| < \infty$ of Theorem \ref{theoconverge}, the ends of the rotation solitons of this surface diverge.

        \begin{figure}[H]
            \begin{minipage}[b]{0.45\linewidth}
            	\includegraphics[scale=0.4]{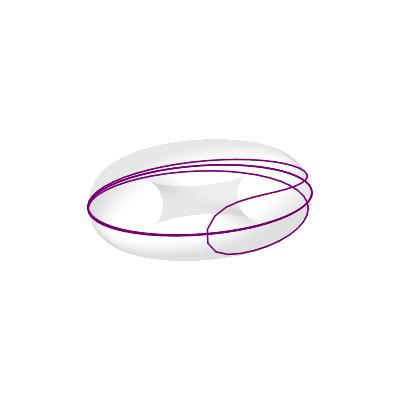}\centering
            	\caption{Initial conditions: $u(0) = \frac{\pi}{4}$, $v(0) =  \frac{\pi}{4}$, $u'(0) = 0$, $v'(0) = 1$, $\xi'(0)=\frac{1}{2}$.}
            	\label{S0a}
            \end{minipage} \hfill
            \begin{minipage}[b]{0.45\linewidth}
            	\includegraphics[scale=0.4]{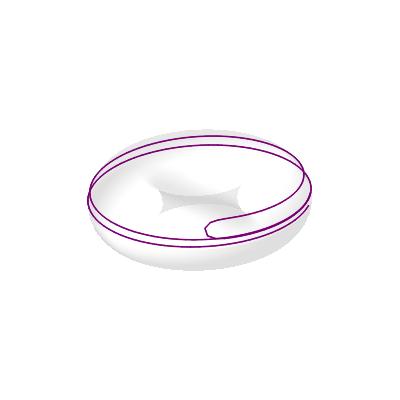}\centering
            	\caption{Initial conditions: $u(0) = \frac{\pi}{4}$, $v(0) =  \frac{\pi}{4}$, $u'(0) = 0$, $v'(0) = 1$, $\xi'(0)=2$.}
            	\label{S0b}
            \end{minipage}
        \end{figure}        
        
        \begin{figure}[H]
            \begin{minipage}[b]{0.45\linewidth}
            	\includegraphics[scale=0.4]{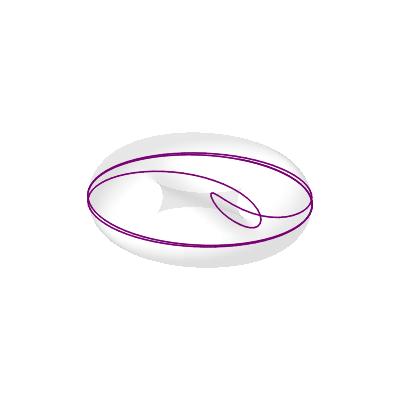}\centering
            	\caption{Initial conditions: $u(0) = \pi$, $v(0) =  \frac{\pi}{4}$, $u'(0) = 1$, $v'(0) = 1$, $\xi'(0)=\frac{1}{2}$.}
            	\label{S0c}
            \end{minipage} \hfill
            \begin{minipage}[b]{0.45\linewidth}
            	\includegraphics[scale=0.4]{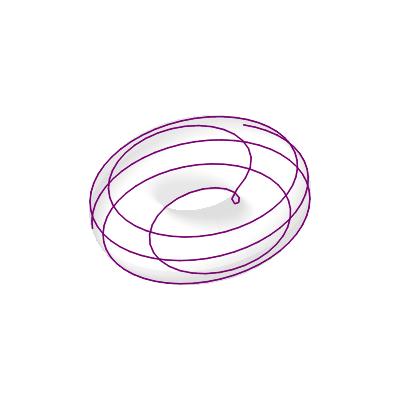}\centering
            	\caption{Initial conditions: $u(0) = \pi$, $v(0) =  \frac{\pi}{4}$, $u'(0) = 1$, $v'(0) = 1$, $\xi'(0)=2$.}
            	\label{S0d}
            \end{minipage}
        \end{figure}     

        \begin{figure}[H]
            \begin{minipage}[b]{0.45\linewidth}
            	\includegraphics[scale=0.4]{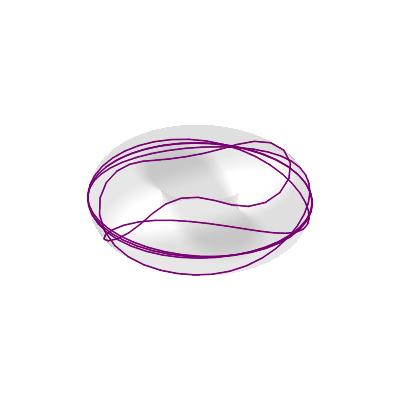}\centering
            	\caption{Initial conditions: $u(0) = 0$, $v(0) =  \pi$, $u'(0) = 1$, $v'(0) = 1$, $\xi'(0)=-0,1$.}
            	\label{Teste 3}
            \end{minipage} \hfill
            \begin{minipage}[b]{0.45\linewidth}
            	\includegraphics[scale=0.4]{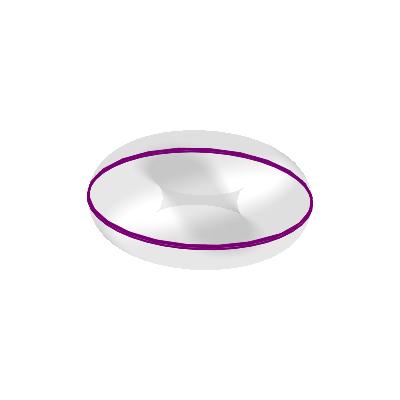}\centering
            	\caption{Initial conditions: $u(0) = 0$, $v(0) =  \pi$, $u'(0) = 1$, $v'(0) = 1$, $\xi'(0)=10$.}
            	\label{Teste 4}
            \end{minipage}
        \end{figure}      
        \begin{figure}[H]
            \begin{minipage}[b]{0.45\linewidth}
            	\includegraphics[scale=0.4]{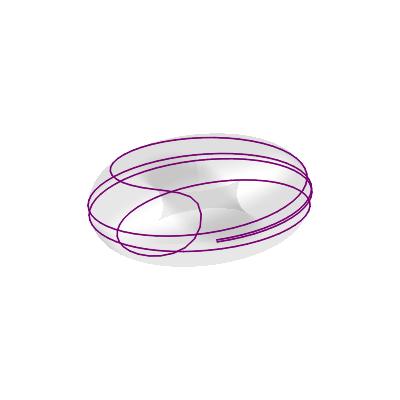}\centering
            	\caption{Initial conditions: $u(0) = \pi$, $v(0) =  \pi$, $u'(0) = 1$, $v'(0) = 0$, $\xi'(0)=1$.}
            	\label{Teste 5}
            \end{minipage} \hfill
            \begin{minipage}[b]{0.45\linewidth}
            	\includegraphics[scale=0.4]{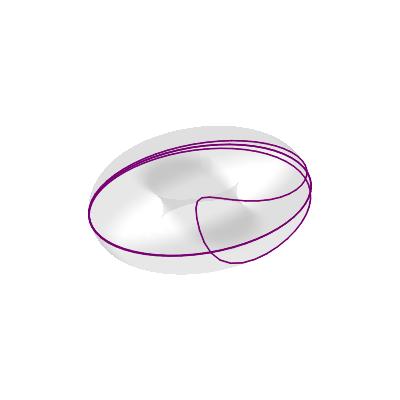}\centering
            	\caption{Initial conditions: $u(0) = \pi$, $v(0) =  \pi$, $u'(0) = 1$, $v'(0) = 0$, $\xi'(0)=-0,5$.}
            	\label{Teste 6}
            \end{minipage}
        \end{figure}

        \begin{figure}[H]
            \begin{minipage}[b]{0.45\linewidth}
            	\includegraphics[scale=0.3]{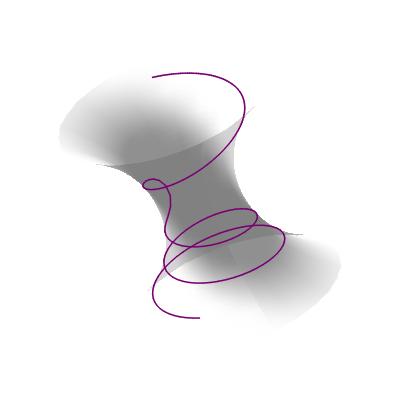}\centering
            	\caption{Initial conditions: $u(0) = 1$, $v(0) =  1$, $u'(0) = 2$, $v'(0) = 1$, $\xi'(0)=1$.}
            	\label{cat1}
            \end{minipage} \hfill
            \begin{minipage}[b]{0.45\linewidth}
            	\includegraphics[scale=0.3]{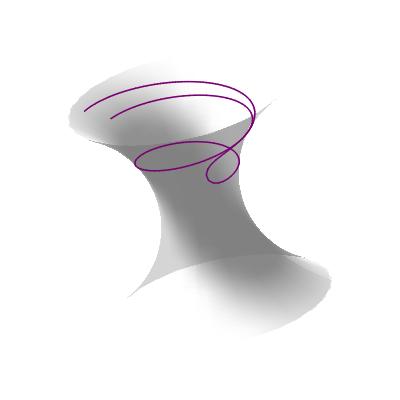}\centering
            	\caption{Initial conditions: $u(0) = 2$, $v(0) =  1$, $u'(0) = -1$, $v'(0) = 1$, $\xi'(0)=1$.}
            	\label{cat2}
            \end{minipage}
        \end{figure} 

        \begin{figure}[H]
            \begin{minipage}[b]{0.45\linewidth}
            	\includegraphics[scale=0.3]{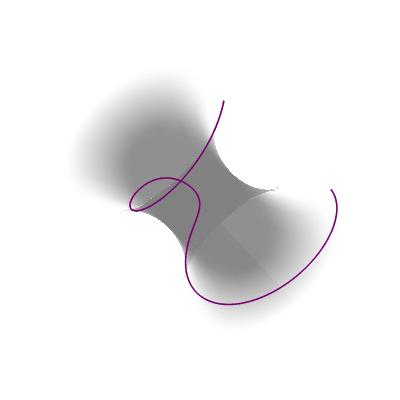}\centering
            	\caption{Initial conditions: $u(0) = 0$, $v(0) =  0$, $u'(0) = \frac{1}{5}$, $v'(0) = \frac{9}{10}$, $\xi'(0)=\frac{1}{2}$.}
            	\label{cat3}
            \end{minipage} \hfill
            \begin{minipage}[b]{0.45\linewidth}
            	\includegraphics[scale=0.3]{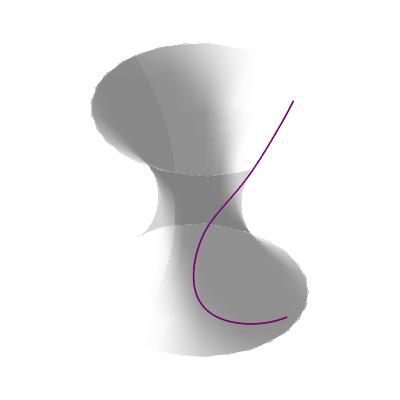}\centering
            	\caption{Initial conditions: $u(0) = 4$, $v(0) =  0$, $u'(0) = \frac{1}{4}$, $v'(0) = \frac{55}{100}$, $\xi'(0)=\frac{1}{5}$.}
            	\label{cat4}
            \end{minipage}
        \end{figure}

             \begin{acknowledgement}
       The authors are grateful to Professor Jo\~ao Paulo dos Santos for fruitful conversations about this work. Benedito Leandro was partially supported by CNPq/Brazil Grant 403349/2021-4 and 303157/2022-4. Rafael Novais was partially supported by PROPG-CAPES [Finance Code 1811476].
        \end{acknowledgement}
    

\begin{thebibliography}{99}

     \bibitem{alias}{L. J. Al\'ias; J. H. de Lira; M. Rigoli} - \textit{Mean curvature flow solitons in the presence of conformal vector fields.} J. Geom. Anal. (2020) 30:1466-1529. 
    	
    	\bibitem{CARMO GD}{M. P. do Carmo} - \textit{Differential geometry of curves and surfaces}: revised and updated second edition. Courier Dover Publications, (2016). MR3837152

    	\bibitem{cao2003}{F. Cao} - \textit{Geometric Curve Evolution and Image Processing.} Lecture Notes in Mathematics, 1805. Springer-Verlag, Berlin, 2003. x+187 pp. ISBN: 3-540-00402-5. MR1976551
    	
    	\bibitem{Colding}{T. H. Colding, W. P. Minicozzi II; E. K. Pedersen} - \textit{Mean Curvature Flow.} Bull. AMS. 52.2 (2015): 297-333. MR3312634

    	\bibitem{fecko}{M. Fecko} - \textit{Differential geometry and lie groups for physicists.} Cambridge University Press (2006). MR2260667 

    	\bibitem{Gage}{M. E. Gage} - \textit{Curve shortening on surfaces.} Annales scientifiques de l'\'Ecole Normale Sup\'erieure 23.2 (1990): 229-256. MR1046497

    	\bibitem{Garcke}{H. Garcke; R. Nürnberg} - \textit{Numerical approximation of boundary value problems for curvature flow and elastic flow in Riemannian manifolds.} Numer. Math. 149, 375–415, (2021). MR4332794

    	\bibitem{halldorsson2012}{H. P. Halldorsson} - \textit{Self-similar solutions to the curve shortening flow.} Transactions AMS. (2012): 5285-5309. MR2931330

    	\bibitem{halldorsson2015}{H. P. Halldorsson} - \textit{Self-similar solutions to the mean curvature flow in the Minkowski plane $\mathbb{R}^{1,1}$.} J. Reine Angew. Math. 704 (2015), 209-243. MR3365779


    	\bibitem{Hungerbuhler} {N. E. Hungerbuhler; B. Roost} - \textit{Mean curvature flow solitons. Analytic aspects of problems in Riemannian geometry: elliptic PDEs, solitons and computer imaging}. 129-158, S\'emin. Congr., 22, Soc. Math. France, Paris, 2011. MR3060452

    	\bibitem{Hiuri1}{H. dos Reis; K. Tenenblat} - \textit{Soliton solutions to the curve shortening flow on the sphere.} Proc. Amer. Math. Soc. 147 (2019), no. 11, 4955–4967. MR4011527

    	\bibitem{Silva2021}{F. N. da Silva; K. Tenenblat} - \textit{Self-Similar Solutions to the Curvature Flow and its Inverse on the 2-dimensional Light Cone.} arXiv:2109.03677.
    	
    	\bibitem{Silva2}{F. N. da Silva; K. Tenenblat} - \textit{Soliton solutions to the curve shortening flow on the 2-dimensional hyperbolic plane.} Rev. Mat. Iberoam. (2022), published online first. DOI 10.4171/RMI/1343.
    	
        \bibitem{slotine}{J-J. E. Slotine; W. Li} - \textit{Applied nonlinear control.} (1991) Prentice Hall, New Jersey.


    	\bibitem{woolgar}{E. Woolgar; R. Ran} - \textit{Self-similar curve shortening flow in hyperbolic $2$-space.} Proc. Amer. Math. Soc. 150 (2022), no. 3, 1301-1319. MR4375723
    \end{thebibliography}
\end{document}